\documentclass[12pt,draftcls,onecolumn]{IEEEtran}
\usepackage{amsmath}
\usepackage{amsfonts}
\usepackage{graphicx}
\usepackage{amssymb}

\newtheorem{condition**}{A*}
\newtheorem{condition***}{C*}
\newtheorem{condition*}{C}

\newtheorem{proposition}{Proposition}[section]

\newtheorem{definition}{Definition}[section]
\newtheorem{theorem}{Theorem}[section]
\newtheorem{lemma}{Lemma}[section]
\newtheorem{remark}{Remark}[section]

\begin{document}
%
\title{A Class of Mean-Field Games with Optimal Stopping and its Inverse Problem}

\author{ Jianhui Huang, Tinghan Xie$^*$\thanks{J. Huang and T. Xie are with the Department of Applied Mathematics, The Hong Kong Polytechnic University, Hong Kong
(majhuang@polyu.edu.hk; tinghan.xie@connect.polyu.hk).}
\thanks{The author acknowledges the financial support from RGC grant 15301119 and 15307621, and also is grateful to the helpful comments from Minyi Huang.}}

\maketitle

\begin{abstract}This paper revisits the well-studied \emph{optimal stopping} problem but within the \emph{large-population} framework. In particular, two classes of optimal stopping problems are formulated by taking into account the \emph{relative performance criteria}. It is remarkable the relative performance criteria, also understood by the \emph{Joneses preference}, \emph{habit formation utility}, or \emph{relative wealth concern} in economics and finance (e.g., \cite{Abel}, \cite{BC}, \cite{CK}, \cite{DeMarzo}, \cite{ET}, \cite{Gali}, etc.), plays an important role in explaining various decision behaviors such as price bubbles. By introducing such criteria in large-population setting, a given agent can compare his individual stopping rule with the average behaviors of its cohort. The associated mean-field games are formulated in order to derive the decentralized stopping rules. The related consistency conditions are characterized via some coupled equation system and the asymptotic Nash equilibrium properties are also verified. In addition, some \emph{inverse} mean-field optimal stopping problem is also introduced and discussed.
\end{abstract}


%
\IEEEpeerreviewmaketitle

\section{Introduction}
In recent years, the study of dynamic optimization of stochastic large-population system has attracted persistent and increasing research attentions.
The agents (or players) in large population system are individually insignificant but collectively impose an important impact on each agent. In mathematical modeling, this feature can be characterized by the state-average coupling structure in the individual dynamics and cost functionals. The large-population systems can be found their applications in various different domains (e.g., engineering, social science, economics and finance, operational research and management, etc.). The interested readers may refer \cite{H10}, \cite{hcm07}, \cite{hmc06} and the reference therein for more details of the solid backgrounds of large-population system.

Regarding the controlled large-population system, it is infeasible for a given agent to collect the ``central" or ``global" information of all agents due to the dimensionality difficulty and complex interactions among the agents. Alternatively, it is more feasible and effective to study the decentralized strategies which depend on the local information only. By ``local information", we mean the optimal control regulator for a given agent, is designed on its own individual state and some quantity which can be computed in an off-line manner. Along this research line, one efficient method is the mean-field game (MFG for short) (e.g., \cite{LL07} or \cite{hcm07}) which fundamental idea is to approximate the initial large-population control problem by its limiting problem through some mean-field term (i.e., the asymptotic limit of state-average). As a consequence, one may design a set of decentralized strategies for the large but finite population system originally analyzed, where each agent needs only to know its own state information and mass-effect computed off-line. Furthermore, it is possible to verify the $\epsilon-$Nash equilibrium property for the derived decentralized strategies in which each individual agent will bear the optimality loss in the level $\epsilon$ depending on the population size. Some recent literature can be found in \cite{B12}, \cite{GLL10}, \cite{hcm07}, \cite{LZ08}, \cite{hhl}, for the study of mean-field games; \cite{HCM12} for cooperative social optimization; \cite{H10}, \cite{NH12} and \cite{NC13} and references therein for models with a major player, etc.

Our paper differs from the existing MFG literature in three main respects. (i) First, we investigate the large population system by its optimal stopping problems instead the already extensively studied optimal control problems. Undoubtedly, the optimal stopping problem itself is not new as it has already been explored in various different fields (for instance, the optimal investment with stopping decision from mathematical finance). Consequently, there exist considerable literature of optimal stopping and its applications. The interested readers are recommended to some monographs (e.g., \cite{o}, \cite{p}, \cite{s}, etc.) and the reference therein for more details. However, to our best knowledge, there has no formal discussion to optimal stopping problems pertinent to large-population framework. For instance, to discuss the behavior interaction of stopping rules and its effects imposed on market investors, as well as the resulting equilibrium. (ii) Second, we introduce the so-called ``relative performance" criteria into our stopping problems within large-population framework. Specially, we formulate and analyze two classes of optimal stopping problems in which a given agent takes into account the relative performance with other agents (in its ``ratio-habit" representation and convex combination). As a sequel, each individual agent aims to solve an optimal stopping problem which is interrelated to others via the terminated state-average of underlying large-population system. (iii) Third, we propose the  \emph{inverse} mean-field optimal stopping problem where an additional manager agent is introduced which can optimally design the payoff functional such that the corresponding stopping rules adopted by all small agents will meet some preferred statistical properties. Recently, there arise a variety of papers concerning to mean-field stopping time but in different setups such as \cite{nutz} for mean-field stopping with continuum agents, and \cite{bgdr} for relaxed optimal stopping.

The introduction of \emph{relative performance} plays some significant role in our problem formulation and the analysis followed. Due to this reason, we would like to present few more words to illustrate its real meanings. First, the relative performance is extensively recognized in economics and decision analysis. For example, it is well documented in \cite{Gomez} that the agents in an economy will manifest substantial preferences exogenously defined over their own consumption as well as the contemporaneous average consumption of a reference group. Such reference group is also called the agent's countrymen which corresponds to the large-population in our setup, and these preferences are often termed in economics by ``\emph{keeping up with the Joneses}" (KUJ) (\cite{Abel}, \cite{CK}, \cite{Gali}, \cite{Gomez}). For short, the \emph{Joneses utility or preference}. Second, it is worth noting that the relative performance or relative wealth concern is also addressed in some optimal investment problem from mathematical finance (e.g. \cite{ET}). As stated in \cite{ET}, ``a return of $5\%$ during the crisis is not equivalent to the same return during a financial bubble". Therefore, it is necessary to study the interactions among all investors (agents) based on the simplified comparison of the performance of their competitors (peers). The intuition behind is that the human being tends to compare themselves to their peers (or, cohorts) and this effect is also supported by empirical studies in economics and social science. Third, the relative performance also arise naturally when considering the benchmark index tracking or the habit persistence preference. In particular, in case there exists some market systematic risk. More reference literature of relative performance can be found in \cite{Veblen} for its sociological part, and \cite{Abel}, \cite{BMW}, \cite{CMP}, \cite{DeMarzo}, \cite{Gali}, \cite{Gomez} for economic part in which the discrete-time model and framework are considered.

Motivated by the above discussion, we study the large population optimal stopping by considering the related \emph{relative performance}. Roughly speaking, we concentrate on a large population system with $N$ symmetric individual agents. These small agents execute stopping decisions by comparing the stopping rules applied by other peers. Specially, each agent takes into account a \emph{convex combination} of his stopping performance (with the convex weight $\theta \in (0, 1)$ and the relative concern to the average behaviors with weight $1-\theta$). In this way, there arises some natural interactions among agents and therefore leads to some large-population optimal stopping games. This corresponds to the economy in which the equilibrium consequences of an economy populated by agents with keeping up with the Joneses preferences. We are more interested to the asymptotic behavior when $N$ tends to infinity and in this case, the situation can be considerably simplified through the mean-field games (\cite{LL07}). Based on it, we can derive the decentralized stopping rules as well as the related consistency condition.

The rest of this paper is organized as follows. In Section \uppercase\expandafter{\romannumeral2}, we formulate two classes of mean-field optimal stopping problems in a large-population framework. The first class arises from the \emph{best time to sell} problem whereas the second class is related to the \emph{valuation of natural resources} problem. Both classes consider the \emph{relative performance} thus the individual agents are coupled via their payoff functionals. Section \uppercase\expandafter{\romannumeral3} aims to study the consistency condition or Nash certainty equivalence (NCE) system of the mean-field optimal stopping problems. Section \uppercase\expandafter{\romannumeral4} studies the relevant asymptotic $\epsilon-$Nash equilibrium property. Section \uppercase\expandafter{\romannumeral5} discusses the \emph{inverse mean-field optimal stopping} problem. Section \uppercase\expandafter{\romannumeral6} provides the conclusion of our work.

\section{Formulation}

We consider a stochastic large population system with $N$ negligible agents, denoted respectively by $\mathcal{A}_i, 1 \leq i \leq N. $ The state dynamics of all agents are given on a complete probability space $(\Omega,
\mathcal F, \mathbb{P})$ on which a standard $N$-dimensional Brownian motion $\{W_i(t),\ 1\le i\leq N\}_{t \geq 0}$ is defined. We denote by  $\{\mathcal{F}^{i}_t\}_{t \geq 0}$ the natural filtration generated by $i^{th}$ Brownian motion $W_i$ and $\mathcal{F}^{i}_0$ contains all $\mathbb{P}-$null sets of $\mathcal F.$ Consequently, $\mathcal F_t=\bigvee_{i=1}^N\mathcal F^{i}_t$ the full information of large population system up to time $t$. We denote $\mathcal{S}$ the set of all stopping times of the filtration $\{\mathcal{F}_t\}_{t \geq 0},$ $\mathcal{S}^{i}$ the set of all stopping times of the filtration $\{\mathcal{F}_t^{i}\}_{t \geq 0}.$ Hereafter, $\mathcal{T}=(\tau_1, \cdots, \tau_{N})$ represents the set of stopping strategies of all $N$ agents; $\tau_{-i}=(\tau_1, \cdots, \tau_{i-1},$ $\tau_{i+1}, \cdots \tau_{N})$ the stopping strategies except $\mathcal{A}_i$. Without regarding ordinal notation, we write $\mathcal{T}=(\tau_i, \tau_{-i}).$

\subsection{One standard optimal stopping problem}
To start, we first recall the following well-studied optimal stopping problem for single agent which is named \emph{best time to sell}.
The agent's state dynamics is given by a geometric Brownian motion (GBM):
\begin{equation}\label{1e1}\left\{
\begin{aligned}
&dx_i(t)=\alpha x_i(t)dt+\sigma x_i(t)dW_i(t),\\
&x_i(0)=x
\end{aligned}\right.\end{equation}where $\alpha \in \mathbb{R}, \sigma>0$ are respectively the return and volatility rate. The initial endowment $x>0$. Suppose $x_i$ denotes the value process of an asset, and the owner of this asset may sell it at any time, but has to pay a fixed transaction fee $K>0.$ In this case, the owner's objective is to solve the following optimal stopping problem:\begin{equation}\nonumber
\max_{\tau_i \in \mathcal{S}^{i}}J_i(\tau_i)=\max_{\tau_i \in \mathcal{S}^{i}}\mathbb{E}\left\{e^{-\beta\tau_i}(x_i(\tau_i)-K)\right\}.
\end{equation}Here, $\beta>0$ is the discounted factor and the payoff (gain) functional can be rewritten as:\begin{equation}\label{1e2}
J_i(\tau_i)=\mathbb{E}\left\{e^{-\beta\tau_i}x_i(\tau_i)-e^{-\beta\tau_i}K\right\}.
\end{equation}Note that the transaction cost $K$ is fixed and should be same for all individual agents because it is assigned by the market regulator or planner. This observation provides the key linkage when we introduce the inverse mean-field optimal stopping problem later. Moreover, we assume all agents in our large population system are symmetric or statistical ``homogeneous" in that they own the same coefficients $(\alpha, \sigma, \beta, K)$.

\subsection{Relative performance in large-population system}
First, we should note that the following utility and preference functional is introduced in \cite{Gomez}: \begin{equation}\nonumber u(c, C)=\frac{c^{1-\alpha}}{1-\alpha}C^{\gamma \alpha}\end{equation}where $c$ denotes the agents' individual consumption and $C$ denotes the consumption average or per capital consumption in given reference group; $\alpha>0$ is the constant relative risk-aversion (CRRA) coefficient; $0 \leq \gamma<1$ is called the ``\emph{Joneses parameter}" which increases the weight of average consumption per capita $C$
and makes the individual's marginal consumption more valuable since it helps individual agent to keep
up with the peers. More insights can be gained if we further assume $\gamma=\frac{\alpha-1}{\alpha}$ for $\alpha>1$ thus \begin{equation}\label{ratiohabit} u(c, C)=\frac{(c/C)^{1-\alpha}}{1-\alpha}\end{equation}which becomes the standard ``\emph{ratio-habit}" representation proposed in \cite{Abel}.
Moreover, $(1-\alpha)$ represents the average consumption elasticity of marginal utility. In addition, we remark \eqref{ratiohabit} specifies the the keeping up with the Joneses preference in exogenous way, and it provides one example of relative performance in economic literature.

Second, given the above standard optimal stopping problem $\eqref{1e1}$, $\eqref{1e2}$ and the relative performance from Joneses preference \eqref{ratiohabit}, as well as the striking large-population feature, it is natural to consider the relative performance based on $e^{-\beta\tau_i}x_i(\tau_i),$ the principal term in payoff functional \eqref{1e2}. As the response, we construct the following criteria term\begin{equation}\label{tax}\frac{e^{-\beta\tau_i}x_i(\tau_i)}{l_1+l_2\left(\frac{1}{N}\sum\limits_{j=1}^Ne^{-\beta\tau_j}x_j(\tau_j)\right)}
\end{equation}which is the ratio between the individual discounted truncated state $e^{-\beta\tau_i}x_i(\tau_i)$ and some affine function upon the average $\frac{1}{N}\sum\limits_{j=1}^Ne^{-\beta\tau_j}x_j(\tau_j)$ for $l_1, l_2>0.$ Moreover, $l_2$ can stand for the degree of relative concern to keep us with their peers (as discussed in \cite{Gomez}).

\begin{remark}First, the criteria \eqref{tax} can be viewed as some modified version of ``ratio-habit" term \eqref{ratiohabit} by considering the optimal stopping outcome (note that $e^{-\beta\tau_i}x_i(\tau_i)$ denotes the discounted stopped state for single agent). Second, if let $l_1=0,$ \eqref{tax} can be connected to the Boltzmann-Gibbs distribution as follows:\begin{equation}\nonumber\begin{aligned}
&\frac{e^{-\beta\tau_i}x_i(\tau_i)}{l_2\left(\frac{1}{N}\sum\limits_{j=1}^Ne^{-\beta\tau_j}x_j(\tau_j)\right)}
=\underbrace{\frac{e^{-\beta\tau_i}x_i(\tau_i)}{\sum\limits_{j=1}^Ne^{-\beta\tau_j}x_j(\tau_j)}}_{\text{Boltzmann-Gibbs distribution}}\cdot\underbrace{l_2^{-1}N}_{\text{Scaled large population size or market capacity}}
\end{aligned}\end{equation}There are some literature to discuss the Boltzmann-Gibbs distribution in economics and wealth allocation such as the agent-based DSGE (dynamic stochastic general equilibrium) model (\cite{BM})
\end{remark}

\begin{remark} From panel data analysis in statistics, the average term $\frac{1}{N}\sum\limits_{j=1}^Ne^{-\beta\tau_j}x_j(\tau_j)$ in \eqref{tax} can be viewed as the cross-section data (or, longitudinal data) built on the truncated terminal wealth $\{x_j\}_{j=1}^{N}$ over all terminated times $\{\tau_j\}_{j=1}^{N}$.
Note that different agents will execute different stopping times even though they may have the same distribution.
Moreover, the affine function $l_1+l_2(\frac{1}{N}\sum\limits_{j=1}^Ne^{-\beta\tau_j}x_j(\tau_j))$ can represent some market primitive. For example, the proportional (revenue) tax rate which consists of two parts: the tax basis $l_1>0$ and the surplus tax part which is monotonic with the industry average $\frac{1}{N}\sum\limits_{j=1}^Ne^{-\beta\tau_j}x_j(\tau_j)$ for $l_2>0.$ Therefore, the criteria \eqref{tax} actually represents some after-tax stopped value based on the tax numeraire $l_1+l_2(\frac{1}{N}\sum\limits_{j=1}^Ne^{-\beta\tau_j}x_j(\tau_j))$.
\end{remark}

\subsection{Convex combination of relative performance}

Now we present more details of the convex combination which is important to construct our functional.
One example of convex functional can be found in is recent paper \cite{ET} where\begin{equation}\label{convex}\mathbb{E}U_i\left((1-\lambda)X^{i}_{T}+\lambda\left(X^{i}_{T}
-\frac{1}{N}\sum_{j=1}^{N}X^{j}_{T}\right)\right)\end{equation}where $U_i$ is the utility function of investor $\mathcal{A}_i$ while $\frac{1}{N}\sum_{j=1}^{N}X^{j}_{T}$ is the average behavior of all investors. It can be understood by the investors' \emph{relative wealth concern} which may help us to explain some financial bubble and negative risk premium. First, agents may display ``external habit formation" (EHF) in their
preferences. In this case, the utility of the investors depends on the wealth of their peers (the ``Joneses")
and investors bias their portfolio holdings towards securities which are correlated with the wealth of their
peers so as to ``keeping up with the Joneses". Second, relative wealth concerns may also arise endogenously,
without assuming EHF preferences. \cite{DeMarzo} shows that individuals with
standard preferences might care about the wealth of their peers because competition for non-diversifiable
assets in limited supply drives their price up; if investors cannot compete in wealth with their peers
they might be left out of the market.

Another supporting example of convex combination is given in \cite{NCMH} where the LQG mean-field game with leader-follower control is discussed.
Specially, \cite{NCMH} introduces a convex combination of cost functional based on a trade-off
between moving towards a common reference trajectory and
keeping cohesion of the flock of leaders by also tracking their
centroid: \begin{equation}\label{convex2}
\phi^{L}(z^{L, N_{L}})(\cdot):=\lambda h(\cdot)+(1-\lambda)z^{L, N_{L}}(\cdot)\end{equation}where $\lambda \in (0,1)$ is the scalar convex index,
$h(\cdot)$ is some reference trajectory while $z^{L, N_{L}}(\cdot):=\frac{1}{N_{L}}\sum_{i=1}^{N_{L}}z_{i}^{L}$ is the centroid of leader group.

\subsection{Optimal stopping of MFG with relative performance (I)}

Motivated by the above relative preference functional \eqref{ratiohabit}, \eqref{tax} and convex functional \eqref{convex}, \eqref{convex2}, we introduce the following optimal stopping problem with relative payoff functional:
\begin{equation}\label{1e4}
\max_{\tau_i,\tau_{-i} \in \mathcal{S}}\mathcal{J}_i(\tau_i,\tau_{-i})=\max_{\tau_i,\tau_{-i} \in \mathcal{S}}\mathbb{E}\left\{{C}_{\theta}(\tau_i,\tau_{-i})-e^{-\beta\tau_i}K\right\}
\end{equation}where the convex combination of relative performance ${C}_{\theta}(\tau_i,\tau_{-i})$ is given by\begin{equation}\label{e6}
{C}_{\theta}(\tau_i,\tau_{-i})=\theta \underbrace{e^{-\beta\tau_i}x_i(\tau_i)}_{\text{absolute performance}}+(1-\theta)\underbrace{\frac{e^{-\beta\tau_i}x_i(\tau_i)}{l_1+l_2\left(\frac{1}{N}\sum\limits_{j=1}^
Ne^{-\beta\tau_j}x_j(\tau_j)\right)}}_{\text{affine relative performance}}.
\end{equation}In its full details, the payoff functional for agent $\mathcal{A}_i$ can be written as follows\begin{equation}\label{1e5}
\mathcal{J}_i(\tau_i,\tau_{-i})=\mathbb{E}\left\{\theta e^{-\beta\tau_i}x_i(\tau_i)+(1-\theta)\frac{e^{-\beta\tau_i}x_i(\tau_i)}{l_1+l_2\left(\frac{1}{N}\sum\limits_{j=1}^Ne^{-\beta\tau_j}x_j(\tau_j)\right)}-e^{-\beta\tau_i}K\right\}.
\end{equation}Here, $\theta\in[0,1]$ is the convex weight index. The stopping time $\tau_i \in \mathcal{S},$ the set of all stopping times of the filtration $\{\mathcal{F}_t\}_{t \geq 0}$. We write the above relative performance functional by $\mathcal{J}_i(\tau_i,\tau_{-i})$ to emphasize its dependence on both $\tau_i$ and $\tau_{-i}$ due to the weakly coupling structure in payoff functionals. More explanations to the convex combination are as follows.
\begin{remark}
(\romannumeral1) If $\theta=0$, \eqref{1e5} becomes
\begin{equation}\nonumber
\mathcal{J}_i(\tau_i,\tau_{-i})=\mathbb{E}\left\{ e^{-\beta\tau_i} \Bigg[ \frac{x_i(\tau_i)}{l_1+l_2\left(\frac{1}{N}\sum\limits_{j=1}^Ne^{-\beta\tau_j}x_j(\tau_j)\right)}-K\Bigg]\right\}.
\end{equation}
In this case, only the relative performance by comparing the selling values among all agents is considered.

(\romannumeral2) If $\theta=1$, \eqref{1e5} takes the following form\begin{equation}\nonumber
\mathcal{J}_i(\tau_i,\tau_{-i})=\mathbb{E}\Bigg\{e^{-\beta\tau_i}\Big( x_i(\tau_i)-K\Big)\Bigg\}
\end{equation}
which is the standard performance criteria for classical \emph{best time to sell} problem.
\end{remark}

By taking different values of $\theta$ in \eqref{1e5}, the individual agents can maximize the expectation of his trade-off between the classical criteria and the relative performance criteria. This functional is based on a balance between the absolute value of selling and the cohesion of other peers. Now, we formulate the following large-population optimal stopping problem.
\\

\textbf{Problem (I)}
Find a stopping strategy set $\mathcal{T}=(\tau_1, \cdots, \tau_{N})$ to maximize
$\mathcal{J}_i(\tau_i,\tau_{-i})$ where $\tau_i \in \mathcal{S}^{i}$ for $1 \leq i \leq N.$\\

\begin{remark}Note that here we consider $\tau_i$ to be taken from $\mathcal{S}^{i}$, the set of all stopping times of the filtration $\{\mathcal{F}^{i}_t\}_{t \geq 0}$. In that case, we call $\tau_i$ the \emph{decentralized} stopping rules as it need only to be adapted to the filtration generated by $\mathcal{A}_i.$\end{remark}

\subsection{Optimal stopping of MFG with relative performance (II)}

Now, we present another optimal stopping problem arising from large population system. We still consider a geometric Brownian motion (GBM) for individual agent:
\begin{equation}\nonumber\left\{
\begin{aligned}
&dx_i(t)=\alpha x_i(t)dt+\sigma x_i(t)dW_i(t),\\
&x_i(0)=x.
\end{aligned}\right.\end{equation}We consider a firm producing some natural resources (crude oil, natural gas, etc) with the market price process given by $x_i$. The
running profit of this production is given by a nondecreasing function $f$ depending on the market price. The given firm may decide at any time to stop the production at a fixed constant cost
$K$. Therefore, the real option value of the firm can be measured by the following optimal stopping problem:\begin{equation}\max_{\tau_i \in \mathcal{S}^{i}}{J}_i(\tau_i)=\max_{\tau_i \in \mathcal{S}^{i}}\mathbb{E}\left\{\int_0^{\tau_i} e^{-\beta t}f(x_i(t))dt-e^{-\beta\tau_i}K\right\}.
\end{equation}We assume $f(\cdot)>0$ is nondecreasing and Lipschitz continuous function. Considering the relative performance, similar to \eqref{1e5}, we can introduce the following payoff functional
\begin{equation}\label{1e8}
\mathcal{J}_i(\tau_i,\tau_{-i})=\mathbb{E}\left\{{C}_{\theta}(\tau_i,\tau_{-i})-e^{-\beta\tau_i}K\right\}
\end{equation}where$${C}_{\theta}(\tau_i,\tau_{-i})=
\Bigg[\theta \int_0^{\tau_i} e^{-\beta t}f(x_i(t))dt+\frac{(1-\theta)\int_0^{\tau_i}e^{-\beta t}f(x_i(t))dt}{l_1+l_2\left(\frac{1}{N}\sum\limits_{j=1}^N \int_0^{\tau_j}e^{-\beta t}f(x_j(t))dt\right)}\Bigg].$$Now, we formulate the following large-population optimal stopping problem.\\

\textbf{Problem (II)}
Find a stopping strategy set $\mathcal{T}=(\tau_1, \cdots, \tau_{N})$ to maximize
$\mathcal{J}_i(\tau_i,\tau_{-i})$ where $\tau_i \in \mathcal{S}^{i}$ for $1 \leq i \leq N.$

\section{Consistency condition of mean-field optimal stopping problem}

\subsection{Consistency condition for Problem (I)}
Now we first consider the consistency condition of Problem (I). Due to the payoff functional \eqref{1e5}, we introduce
\begin{equation}\nonumber
\widetilde{\theta}_1=\widetilde{\theta}_1(N, \mathcal{T}, \{x_j\}_{j=1}^{N}):=\theta+\frac{1-\theta}{l_1+l_2\left(\frac{1}{N}\sum\limits_{j=1}^Ne^{-\beta\tau_j}x_j(\tau_j)\right)}.
\end{equation}As $\widetilde{\theta}_1\neq0$ thus we can define $K'_1=\frac{K}{\widetilde{\theta}_1}$. Then the payoff functional \eqref{1e5} becomes\begin{equation}\nonumber
\mathcal{J}_i(\tau_i,\tau_{-i})=\mathbb{E}\Bigg\{\widetilde{\theta}_1e^{-\beta\tau_i}\Big(x_i(\tau_i)-K'_1\Big)\Bigg\}.
\end{equation}Considering the decentralized optimal stopping $\tau_i \in \mathcal{S}^{i},$ and by law of large numbers, we have
\begin{equation}\nonumber
\lim_{N\rightarrow+\infty}\frac{1}{N}\sum\limits_{j=1}^Ne^{-\beta\tau_j}x_j(\tau_j)=\mathbb{E}\Big(e^{-\beta\tau}x(\tau)\Big)
\end{equation}for some $(\tau, x)$ following the same distribution with $(\tau_j, x_j)$ for $1\leq j\leq N$.
Then we obtain\begin{equation}\label{e8}\left\{
\begin{aligned}
&\bar{\theta}_1:=\lim_{N\rightarrow+\infty}\widetilde{\theta}_1=\theta+\frac{1-\theta}{l_1+l_2\mathbb{E}\Big(e^{-\beta\tau}x(\tau)\Big)},\\
&\bar{K}_1:=\lim_{N\rightarrow+\infty}K'_1=\frac{K}{\bar{\theta}_1}.
\end{aligned}\right.\end{equation}Next, we get the following auxiliary mean-field optimal stopping functional
\begin{equation}\label{e9}
J_i(\tau_i)=\mathbb{E}\Bigg\{\bar{\theta}_1e^{-\beta\tau_i}\Big(x_i(\tau_i)-\bar{K}_1\Big)\Bigg\}=\bar{\theta}_1\mathbb{E}\Bigg\{e^{-\beta\tau_i}\Big(x_i(\tau_i)
-\bar{K}_1\Big)\Bigg\}
\end{equation}
where $\bar{\theta}_1$ and $\bar{K}_1$ are defined in \eqref{e8}.

Now, we formulate the auxiliary \emph{limiting} mean-field optimal stopping problem for Problem (\textbf{I}) (For short, we denote it as (\textbf{LI})).\\

\textbf{Problem (LI)} Find a stopping strategy set $\mathcal{T}=(\tau_1, \cdots, \tau_{N})$ to maximize
$J_i(\tau_i)$ for $1 \leq i \leq N$ where $\tau_i \in \mathcal{S}^{i}.$

Note that the maximization of $J_i(\tau_i)$ is equivalent to the maximization of $\mathbb{E}\{e^{-\beta\tau_i}(x_i(\tau_i)
-\bar{K}_1)\}.$ Introduce the value function$$v(x):=\sup\limits_{\tau_i\in \mathcal{S}^{i}}\mathbb{E}\Bigg\{e^{-\beta\tau_i}\Big(x_i(\tau_i)-\bar{K}_1\Big)\Bigg\}.$$Then from the standard results (e.g., \cite{o}, \cite{p}, \cite{s}), we have
\begin{equation}\nonumber
\beta v(x)-\alpha x v'(x)-\frac{1}{2}\sigma^2x^2v''(x)=0,
\end{equation}
for $x<x^*$ and $v(x)$ should be represented by\begin{equation}\label{valuefunction1}v(x)=Ax^{k_1}+Bx^{k_2}\end{equation}for some $(A, B).$ Here, $(k_1, k_2)$ are solutions of the following quadratic equation:
$$\frac{1}{2}\sigma^{2}k^{2}+(\alpha-\frac{1}{2}\sigma^{2})k-\beta=0.$$
It is checkable that\begin{equation}\label{k12}\left\{
\begin{aligned}
&k_1=\frac{1}{2}-\frac{\alpha}{\sigma^2}-\sqrt{\Big(\frac{1}{2}-\frac{\alpha}{\sigma^2}\Big)^2+\frac{2\beta}{\sigma^2}}<0,\\
&k_2=\frac{1}{2}-\frac{\alpha}{\sigma^2}+\sqrt{\Big(\frac{1}{2}-\frac{\alpha}{\sigma^2}\Big)^2+\frac{2\beta}{\sigma^2}}>1.
\end{aligned}\right.\end{equation}We assume $\beta >\alpha,$ then the \emph{decentralized} optimal stopping rule for $\mathcal{A}_i$ is characterized by
\begin{equation}\label{e11}
\tau_i^{*}=\inf\{t:\ x_i(t)\geq x^*\},\ \ x^*=\bar{K}_1\cdot\frac{k_2}{k_2-1}.
\end{equation}Then we can reformulate as\begin{equation}\nonumber \left\{
\begin{aligned}
&\tau_i^{*}=\inf\{t:\ W_i(t)\geq a't+b'\},\\
&a'=\frac{\sigma}{2}-\frac{\alpha}{\sigma}, \quad b'=\frac{1}{\sigma}\ln\left(\frac{x^*}{x}\right)=\frac{1}{\sigma}\ln\left(\frac{\bar{K}_1}{x}\cdot\frac{k_2}{k_2-1}\right).
\end{aligned}\right.\end{equation}
We can construct $\tau^{*} \sim \tau_i^{*}$ (with the same distribution) by
\begin{equation}\label{e13}
\tau^{*}=\inf\{t:\ W(t)\geq a't+b'\}.
\end{equation}for some Brownian motion $W.$ For $\forall\ \lambda\in \mathbb{R},\ \mathbb{E}\Big(e^{\lambda W_\tau-\frac{1}{2}\lambda^2\tau}\Big)=1$
thus\begin{equation}\label{e15}
\mathbb{E}\Big(e^{-\beta\tau^{*}}\Big)=e^{-\lambda b'}
\end{equation}for $\lambda$ satisfying\begin{equation}\lambda^2-2a'\lambda-2\beta=0.\end{equation}
Note that $\beta>0$ implies $\Delta=4(a')^2+8\beta>0$ thus the solution pair becomes\begin{equation}\left\{
\begin{aligned}
&\lambda_+=a'+\sqrt{(a')^2+2\beta},\\
&\lambda_-=a'-\sqrt{(a')^2+2\beta}.
\end{aligned}\right.\end{equation}Further it is checkable that $\lambda_-=k_1\sigma, \lambda_+=k_2\sigma.$ Notice the stopping rule \eqref{e11} and $x_i(0)=x$, so we assume $x<x^*$ thus $b'=\frac{1}{\sigma}\ln\left(\frac{x^*}{x}\right)>0$. By $\beta>0,\tau\geq 0$, we have\begin{equation}
\mathbb{E}\Big(e^{-\beta\tau^{*}}\Big)=e^{-\lambda_{+} b'}.
\end{equation}It can be calculated that
\begin{equation}\nonumber
l_1+l_2\Big(\frac{K}{\bar{\theta}_1}\cdot\frac{k_2}{k_2-1}\Big)^{1-\frac{\lambda_+}{\sigma}}x^{\frac{\lambda_{+}}{\sigma}}=\frac{1-\theta}
{\bar{\theta}_1-\theta}
\end{equation}thus we have\begin{equation}\nonumber
l_1+l_2\Big(\frac{K}{\bar{\theta}_1}\cdot\frac{k_2}{k_2-1}\Big)^{1-k_2}x^{k_2}=\frac{1-\theta}{\bar{\theta}_1-\theta}.
\end{equation}In particular, if $\theta=0, l_1=0$,
\begin{equation}\label{e17}
\bar{\theta}_1=x^{-1}l_2^{-\frac{1}{k_2}}\Big(\frac{Kk_2}{k_2-1}\Big)^{\frac{k_2-1}{k_2}}.
\end{equation}When $0<\theta\leq1$, the equation
\begin{equation}\nonumber
l_2\bar{\theta}_1^{k_2-1}\Big(\frac{Kk_2}{k_2-1}\Big)^{1-k_2}x^{k_2}+l_1=\frac{1-\theta}{\bar{\theta}_1-\theta}
\end{equation}becomes the consistency condition for $\bar{\theta}_1$.
\begin{proposition}The consistent condition
\begin{equation}\label{nce1}
l_2\bar{\theta}_1^{k_2-1}\Big(\frac{Kk_2}{k_2-1}\Big)^{1-k_2}x^{k_2}+l_1=\frac{1-\theta}{\bar{\theta}_1-\theta}
\end{equation}admits one unique solution.
\end{proposition}

\begin{proof}LHS: note that $k_2>1,$ $f(\bar{\theta}_1):=l_2\bar{\theta}_1^{k_2-1}\Big(\frac{Kk_2}{k_2-1}\Big)^{1-k_2}x^{k_2}+l_1$ is an increasing function of $\bar{\theta}_1$ with $\lim\limits_{\bar{\theta}_1\longrightarrow +\infty}f(\bar{\theta}_1)=+\infty$ and $\lim\limits_{\bar{\theta}_1\longrightarrow \theta^{+}}f(\bar{\theta}_1)<+\infty$
RHS: note that $\bar{\theta}_1>\theta,$ and $g(\bar{\theta}_1):=\frac{1-\theta}{\bar{\theta}_1-\theta}$ is decreasing function of $\bar{\theta}_1$ with $\lim\limits_{\bar{\theta}_1\longrightarrow +\infty}g(\bar{\theta}_1)=0$ and $\lim\limits_{\bar{\theta}_1\longrightarrow \theta^{+}}g(\bar{\theta}_1)=+\infty$
Therefore, the consistency condition equation \eqref{nce1}: $f(\bar{\theta}_1)=g(\bar{\theta}_1)$ always admits one unique solution $\bar{\theta}_1$.\end{proof}
Moreover, we can also calculate $v(x)$ as follows:\begin{equation}\nonumber\begin{aligned}
v(x)&=\sup\limits_{\tau_i\in \mathcal{S}^{i}}\mathbb{E}\Bigg\{e^{-\beta\tau_i}\Big(x_i(\tau_i)-\bar{K}_1\Big)\Bigg\}
=\mathbb{E}\Bigg\{e^{-\beta\tau_i^{*}}\Big(x_i(\tau_i^{*})-\bar{K}_1\Big)\Bigg\}\\&
=\mathbb{E}\{e^{-\beta\tau_i^{*}}(x_i^{*}-\bar{K}_1)\}=(x_i^{*}-\bar{K}_1)\mathbb{E}\{e^{-\beta\tau_i^{*}}\}
\\&=\frac{K}{\bar{\theta}_1}\cdot\frac{1}{k_2-1}e^{\lambda_{+}b'}=\frac{K}{\bar{\theta}_1}\cdot\frac{1}{k_2-1}e^{\frac{1}{\sigma}\lambda_{+}\ln\left(\frac{\bar{K}_1}{x}\frac{k_2}{k_2-1}\right)}
\\&=\frac{1}{k_2}\left(\frac{K}{\bar{\theta}_1}\cdot \frac{k_2}{k_2-1}\right)^{1-k_2}x^{k_2}\\&=\frac{1}{k_2}\frac{1}{(x^{*})^{k_2-1}}x^{k_2}.
\end{aligned}\end{equation}In other words, $A=0, B=\frac{1}{k_2}\frac{1}{(x^{*})^{k_2-1}}$ in value function representation \eqref{valuefunction1}. This result coincide with that of \cite{p}, pp. 106 but note that here $x^{*}$ actually depends on $\bar{\theta}_1$ which is determined by the NCE equation \eqref{nce1}.

\subsection{Consistency condition for Problem (II)}

Now we discuss the consistency condition of Problem (II). Denote\begin{equation}\nonumber
\widetilde{\theta}_2=\widetilde{\theta}_2(N, \mathcal{T}, \{x_j\}_{j=1}^{N}):=\theta+\frac{1-\theta}{l_1+l_2\left(\frac{1}{N}\sum\limits_{j=1}^N\int_0^{\tau_j}e^{-\beta t}f(x_j(t))dt\right)}.
\end{equation}By law of large numbers, we have
\begin{equation}\nonumber
\lim_{N\rightarrow+\infty}\frac{1}{N}\sum\limits_{j=1}^N\int_0^{\tau_j}e^{-\beta t}f(x_j(t))dt=\mathbb{E}\int_0^{\tau}e^{-\beta t}f(x(t))dt
\end{equation}
where $\tau$ follows the same distribution with $\tau_j$, $1\leq j\leq N$ and $x(\cdot)$ satisfies
\begin{equation}\nonumber
dx(t)=\alpha x(t)dt+\sigma x(t)dW(t),\ \ x(0)=x.
\end{equation}
Then we introduce
\begin{equation}\label{E8}\begin{aligned}
\bar{\theta}_2:=\lim_{N\rightarrow+\infty}\widetilde{\theta}_2=\theta+\frac{1-\theta}{l_1+l_2\left(\mathbb{E}\int_0^{\tau}e^{-\beta t}f(x(t))dt\right)}.
\end{aligned}\end{equation}
Define $\bar{K}_2:=\frac{K}{\bar{\theta}_2}.$
Then we get the auxiliary mean-field optimal stopping functional as
\begin{equation}\label{E9}
J_i(\tau_i)=\mathbb{E}\Bigg\{\bar{\theta}_2\int_0^{\tau_i}e^{-\beta t}f(x_i(t))dt-e^{-\beta\tau_i}K\Bigg\}=\bar{\theta}_2\mathbb{E}\Bigg\{\int_0^{\tau_i}e^{-\beta t}f(x_i(t))dt-e^{-\beta\tau_i}\bar{K}_2\Bigg\}.
\end{equation}The maximization of $J_i(\tau_i)$ is equivalent to the maximization of $\mathbb{E}\Bigg\{\int_0^{\tau_i}e^{-\beta t}f(x_i(t))dt-e^{-\beta\tau_i}\bar{K}_2\Bigg\}$ as $\bar{\theta}_2>0.$
Now, we formulate the auxiliary limiting mean-field optimal stopping problem (for short, we denote it by (\textbf{LII})).\\

\textbf{Problem (LII)}
Find a stopping strategies set $\mathcal{T}=(\tau_1, \cdots, \tau_{N})$ to maximize
$J_i(\tau_i)$ for $1 \leq i \leq N$ where $\tau_i \in \mathcal{S}^{i}.$

Denote $$v(x):=\sup\limits_{\tau_i\in \mathcal{S}^{i}}\mathbb{E}\Bigg\{\int_0^{\tau_i}e^{-\beta t}f(x_i(t))dt-e^{-\beta\tau_i}\bar{K}_2\Bigg\}$$which satisfies
\begin{equation}\nonumber
\beta v(x)-\alpha x v'(x)-\frac{1}{2}\sigma^2x^2v''(x)-f(x)=0
\end{equation}
for $x>x^*.$ The continuation region is specified by
\begin{equation}\nonumber
\tau_i^{*}=\inf\{t:\ x_i(t)\leq x^*\}.
\end{equation}Define the resolvent operator$$\mathcal{R}_{\beta}f(x)=p(x):=\mathbb{E}\int_0^{+\infty}e^{-\beta t}f(x(t))dt.$$
We have the following representation$$v(x)=Ax^{k_1}+Bx^{k_2}+p(x)$$for $(k_1, k_2)$ are given by \eqref{k12}. Moreover, $B=0$ because the linear growth condition and the coefficient $A$ together with the cut-off level $x^{*}$ should be jointly determined by \begin{equation}\label{E12}\left\{\begin{aligned}
&A(x^*)^{k_1}+p(x^*)=-\bar{K}_2,\\
&k_1A(x^*)^{k_1-1}+p'(x^*)=0.
\end{aligned}\right.\end{equation}
By choosing $\tau_i^{*}$, we obtain
\begin{equation}\nonumber \begin{aligned}
v(x)&=Ax^{k_1}+p(x)=\mathbb{E}\int_0^{\tau_i^{*}}e^{-\beta t}f(x_i(t))dt-\bar{K}_2\mathbb{E}\Big(e^{-\beta\tau_i^{*}}\Big).
\end{aligned}\end{equation}Further
\begin{equation}\nonumber \begin{aligned}
\mathbb{E}\int_0^{\tau_i^{*}}e^{-\beta t}f(x_i(t))dt&=Ax^{k_1}+p(x)+\bar{K}_2\mathbb{E}\Big(e^{-\beta\tau_i^{*}}\Big).
\end{aligned}\end{equation}We have
\begin{equation}\label{E16}
\tau_i^{*}=\inf\{t:\ W_i(t)\geq a't+b'\}
\end{equation}where the cut-off condition is given by$$a'=\frac{\sigma}{2}-\frac{\alpha}{\sigma},\ \ b'=\frac{1}{\sigma}\ln\left(\frac{x^*}{x}\right).$$
Still, we have\begin{equation}\nonumber
\mathbb{E}\Big(e^{-\beta\tau_i^{*}}\Big)=e^{-\lambda b'}
\end{equation}for $\lambda a'-\frac{1}{2}\lambda^2=-\beta.$ Noting $x>x^*$, we have $b'=\frac{1}{\sigma}\ln\left(\frac{x^*}{x}\right)<0$. By $\beta>0,\tau\geq 0$, we just choose the negative root $$\lambda_-=a'-\sqrt{(a')^2+2\beta}.$$
Thus it follows $k_1=\frac{\lambda_-}{\sigma}<0$. We have\begin{equation}\nonumber
\mathbb{E}\Big(e^{-\beta\tau_i^{*}}\Big)=e^{-\frac{\lambda_- }{\sigma}\ln\left(\frac{x^*}{x}\right)}=e^{-k_1}\ln\left(\frac{x^*}{x}\right)=\left(\frac{x^*}{x}\right)^{-k_1}.
\end{equation}
Therefore,\begin{equation}\label{E19}
Ax^{k_1}+p(x)+\frac{K}{\bar{\theta}_2}\left(\frac{x^{*}}{x}\right)^{-k_1}=\frac{1}{l_2}\left[\frac{1-\theta}{\bar{\theta}_2-\theta}-l_1\right].
\end{equation}
In summary, \eqref{E12} and \eqref{E19} are called the NCE consistency condition to $(A, x^{*}, \bar{\theta}_2)$.
\begin{theorem}
The NCE consistency condition system
\begin{equation}\label{E20}\left\{
\begin{aligned}
&A(x^*)^{k_1}+p(x^*)=-\frac{K}{\bar{\theta}_2},\\
&k_1A(x^*)^{k_1-1}+p'(x^*)=0,\\
&Ax^{k_1}+p(x)+\frac{K}{\bar{\theta}_2}\left(\frac{x^{*}}{x}\right)^{-k_1}=\frac{1}{l_2}\left[\frac{1-\theta}{\bar{\theta}_2-\theta}-l_1\right]
\end{aligned}\right.\end{equation}admits one unique solution $(A, x^{*}, \bar{\theta}_2)$ where$$p(x):=\mathbb{E}\int_0^{+\infty}e^{-\beta t}f(x(t))dt$$
\end{theorem}
\begin{proof}From first equation of \eqref{E20}, we have $\bar{\theta}_2=-\frac{K}{A(x^{*})^{k_1}+p(x^{*})}$ thus
\begin{equation}\nonumber
Ax^{k_1}+p(x)-(A(x^{*})^{k_1}+p(x^{*}))\left(\frac{x^{*}}{x}\right)^{-k_1}=\frac{1}{l_2}\left[\frac{1-\theta}{-\left(\frac{K}{A(x^{*})^{k_1}+p(x^{*})}+\theta\right)}-l_1\right].
\end{equation}Also, from the second equation of \eqref{E20}, we have $A=-\frac{p'(x^{*})}{k_1(x^{*})^{k_1-1}}$ therefore
\begin{equation}\begin{aligned}\nonumber
&-\frac{p'(x^{*})}{k_1(x^{*})^{k_1-1}}x^{k_1}+p(x)-\left(-\frac{p'(x^{*})}{k_1(x^{*})^{k_1-1}}(x^{*})^{k_1}+p(x^{*})\right)\left(\frac{x^{*}}{x}\right)^{-k_1}\\
&=\frac{1}{l_2}\left[\frac{1-\theta}{-\left(\frac{K}{-\frac{p'(x^{*})}{k_1(x^{*})^{k_1-1}}(x^{*})^{k_1}+p(x^{*})}+\theta\right)}-l_1\right].
\end{aligned}\end{equation}Rearrange the above terms and noting that $x \neq 0,$ we have\begin{equation}\begin{aligned}\nonumber
&\frac{p(x)}{x^{k_1}}-\frac{p(x^{*})}{(x^{*})^{k_1}}=-\frac{1}{l_2}\frac{(1-\theta)}{\left(\frac{Kk_1}{k_1p(x^{*})-p'(x^{*})x^{*}}+\theta\right)}-\frac{l_1}{l_2}.
\end{aligned}\end{equation}Note that $x, l_1, l_2$ are known, thus we get the following equation of $x^{*}$ only:\begin{equation}\begin{aligned}
&-\left(\frac{p(x)}{x^{k_1}}+\frac{l_1}{l_2}\right)+\frac{p(x^{*})}{(x^{*})^{k_1}}=\frac{1}{l_2}\frac{(1-\theta)}{\left(\frac{Kk_1}{k_1p(x^{*})-p'(x^{*})x^{*}}+\theta\right)}.
\end{aligned}\end{equation}We have\begin{equation}\left\{\begin{aligned}
&A=-\frac{p'(x^{*})}{k_1(x^{*})^{k_1-1}},\\
&\bar{\theta}_2=-\frac{K}{A(x^{*})^{k_1}+p(x^{*})}=\frac{Kk_1}{k_1p(x^{*})-p'(x^{*})x^{*}}.\\
\end{aligned}\right.\end{equation}
\end{proof}

\section{Asymptotic analysis of $\epsilon-$Nash equilibrium}We discuss the asymptotic near-optimal property of Problem (I). To start, we first introduce the $\epsilon-$Nash equilibrium property to our mean-field optimal stopping problem. \begin{definition}A set of stopping strategies $\mathcal{T}=(\tau_1, \cdots, \tau_{N})$ for agents $\{\mathcal{A}_i\}_{1 \leq i \leq N},$ is called an $\epsilon-$Nash equilibrium with respect to the payoff functionals $\mathcal{J}_i, 1 \leq i \leq N,$ if there exists $\epsilon>0$ such that for any fixed $1 \leq i \leq N,$ it satisfies that \begin{equation} \mathcal{J}_{i}(\hat{\tau}_i, \tau_{-i})-\epsilon \leq \mathcal{J}_{i}(\tau_i, \tau_{-i})\end{equation}when any alternative optimal stopping $\hat{\tau}_i \in \mathcal{S}^{i}$ is applied by the agent $\mathcal{A}_i.$ \end{definition}Note that here $\mathcal{S}^{i}$ is the set of all stopping times of the filtration $\{\mathcal{F}^{i}_t\}_{t \geq 0}.$
\begin{lemma}For any $\tau \in \mathcal{S},$ we have $\mathbb{E}(e^{-\beta \tau} x_{\tau})^{2}<+\infty.$  \end{lemma}
\begin{proof}In case $\beta>\alpha+\frac{\sigma^{2}}{2},$ we have\begin{equation}\nonumber\begin{aligned}
&\mathbb{E}(\hat{y}_i)^{2}=\mathbb{E}\left(e^{-\beta \hat{\tau}_i} x_i(\hat{\tau}_i)\right)^{2}=\mathbb{E}\left(e^{-2\beta \hat{\tau}_i} \cdot e^{2(\alpha-\frac{1}{2}\sigma^{2})\hat{\tau}_i+2\sigma W_{i}(\hat{\tau}_i)}\right)\\
&=\mathbb{E}\left(e^{2\sigma W_{i}(\hat{\tau}_i)-2\sigma^{2}\hat{\tau}_i} \cdot e^{\left(\sigma^{2}+2(\alpha-\beta)\right)\hat{\tau}_i}\right)\\
& \leq \mathbb{E}\left(e^{2\sigma W_{i}(\hat{\tau}_i)-2\sigma^{2}\hat{\tau}_i}\right)=1<+\infty.
\end{aligned}
\end{equation}In case $\alpha<\beta<\alpha+\frac{\sigma^{2}}{2},$ we have\begin{equation}\nonumber\begin{aligned}
&\mathbb{E}(\hat{y}_i)^{2}=\mathbb{E}\left(e^{-\beta \hat{\tau}_i} x_i(\hat{\tau}_i)\right)^{2}=\mathbb{E}\left(e^{-2\beta \hat{\tau}_i} \cdot e^{2(\alpha-\frac{1}{2}\sigma^{2})\hat{\tau}_i+2\sigma W_{i}(\hat{\tau}_i)}\right)\\
&=\mathbb{E}\left(e^{2\sigma W_{i}(\hat{\tau}_i)-4\sigma^{2}\hat{\tau}_i} \cdot e^{\left(4\sigma^{2}-2\beta+2(\alpha-\frac{1}{2}\sigma^{2})\right)\hat{\tau}_i}\right)\\
& \leq \sqrt{\mathbb{E}\left(e^{2\sigma W_{i}(\hat{\tau}_i)-4\sigma^{2}\hat{\tau}_i}\right)^{2}}\cdot \sqrt{\mathbb{E}e^{2\left(4\sigma^{2}-2\beta+2(\alpha-\frac{1}{2}\sigma^{2})\right)\hat{\tau}_i}}\\& \leq \sqrt{\mathbb{E}e^{\left(4(\alpha-\beta)+6\sigma^{2}\right)\hat{\tau}_i}}<+\infty.
\end{aligned}
\end{equation}\end{proof}
\begin{theorem}The optimal stopping strategy set $\mathcal{T}^{*}=(\tau_1^{*}, \cdots, \tau_{N}^{*})$ is an $\epsilon-$Nash equilibrium where $\tau_i^{*}=\inf\{t: x_i(t) \geq \bar{K}_1 \cdot \frac{k_2}{k_2-1}\}$ and $\bar{K}_1=\frac{K}{\bar{\theta}_1}$ is determined via the NCE equation \eqref{nce1}.
\end{theorem}

\begin{proof}For sake of presentation, we introduce the following notations:\begin{equation}\left\{
\begin{aligned}\nonumber
&y_i(\tau_i):=e^{-\beta\tau_i}x_i(\tau_i), \quad y_{-i}(\tau_{-i}):=\frac{1}{N}\sum\limits_{j=1,j\neq i}^Ne^{-\beta\tau_j}x_j(\tau_j),\\
&y_i^{*}(\tau_i^{*}):=e^{-\beta\tau_i^{*}}x_i(\tau_i^{*}), \quad y_{-i}^{*}(\tau_{-i}^{*}):=\frac{1}{N}\sum\limits_{j=1,j\neq i}^Ne^{-\beta\tau_j^{*}}x_j(\tau_j^{*}).
\end{aligned}\right.\end{equation}When all agents apply the optimal stopping strategies $\mathcal{T}^{*}=\{\tau_1^{*}, \cdots, \tau_N^{*}\},$ the payoff functional in Problem (\textbf{I}) becomes

\begin{equation}\nonumber\begin{aligned}
&\mathcal{J}_i(\tau_i^{*},\tau_{-i}^{*})\\
=&\mathbb{E}\left\{\Bigg[\theta e^{-\beta\tau_i^{*}}x_i(\tau_i^{*})+(1-\theta)\frac{e^{-\beta\tau_i^{*}}x_i(\tau_i^{*})}
{l_1+l_2\left(\frac{1}{N}\sum\limits_{j=1}^Ne^{-\beta\tau_j^{*}}x_j(\tau_j^{*})\right)}\Bigg]-e^{-\beta\tau_i^{*}}K\right\}\\
=&E\left\{e^{-\beta\tau_i^{*}}x_i(\tau_i^{*})\widetilde{\theta}_1-e^{-\beta\tau_i^{*}}K\right\}
\end{aligned}
\end{equation}where\begin{equation}\begin{aligned}&\widetilde{\theta}_1:=\theta+\frac{(1-\theta)}{l_1+l_2\left(\frac{1}{N}\sum\limits_{j=1}^Ne^{-\beta\tau_j^{*}}x_j(\tau_j^{*})\right)}\\&=\theta +\frac{(1-\theta)}{l_1+\frac{l_2}{N}e^{-\beta\tau_i^{*}}x_i(\tau_i^{*})+\frac{l_2}{N}\sum\limits_{j=1,j\neq i}^Ne^{-\beta\tau_j^{*}}x_j(\tau_j^{*})}\end{aligned}\end{equation}Therefore, we have\begin{equation}\widetilde{\theta}_1=\widetilde{\theta}_1(y_i^{*}, y_{-i}^{*})=\theta+\frac{(1-\theta)}{l_1+\frac{l_2y_i^{*}}{N}+l_2y_{-i}^{*}}\end{equation}Therefore,\begin{equation}
\mathcal{J}_i(\tau_i^{*},\tau_{-i}^{*})=\mathbb{E}\left\{y_i^{*}\widetilde{\theta}_1(y_i^{*}, y_{-i}^{*})-e^{-\beta\tau_i^{*}}K\right\}.
\end{equation}On the other hand, the limiting gain functional\begin{equation}J_{i}(\tau_i^{*})=\mathbb{E}\{\bar{\theta}_1y_i^{*}-e^{-\beta \tau_i^{*}}K\}\end{equation}where
$\bar{\theta}_1:=\theta+\frac{1-\theta}{l_1+l_2\mathbb{E}\Big(e^{-\beta\tau}x(\tau)\Big)}$ is determined through the consistency condition \eqref{nce1}. When all the agents apply the decentralized optimal stopping rules $\mathcal{T}^{*},$ applying Cauchy-Schwarz inequality:\begin{equation}\begin{aligned}&|\mathcal{J}_i(\tau_i^{*}, \tau_{-i}^{*})-J_{i}(\tau_i^{*})|=|\mathbb{E}\left[y_i^{*}\left(\widetilde{\theta}_1(y_i^{*}, y_{-i}^{*})-\bar{\theta}_1\right)\right]|\\&=(1-\theta)\mathbb{E}\left[y_i^{*}\left(\frac{1}{l_1+
l_2\frac{y_i^{*}}{N}+l_2y_{-i}^{*}}-\frac{1}{l_1+l_2\mathbb{E}\Big(e^{-\beta\tau}x(\tau)\Big)}\right)\right]\\& \leq (1-\theta)\left(\mathbb{E}(y_i^{*})^{2}\right)^{\frac{1}{2}}\cdot \left(\mathbb{E}\frac{(
l_2\frac{y_i^{*}}{N}+l_2\left(y_{-i}^{*}-\mathbb{E}(e^{-\beta\tau}x(\tau))\right)^{2}}{(l_1+
l_2\frac{y_i^{*}}{N}+l_2y_{-i}^{*})^{2}\big(l_1+l_2\mathbb{E}(e^{-\beta\tau}x(\tau))\big)^{2}}\right)^{\frac{1}{2}}\\
&\leq \frac{2(1-\theta)}{l_1^{2}}\left(\mathbb{E}(y_i^{*})^{2}\right)^{\frac{1}{2}}\cdot \left(\frac{l_2^{2}}{N^{2}}\mathbb{E}(y_i^{*})^{2}+l_2^{2}\mathbb{E}(y_{-i}^{*}-\mathbb{E}(e^{-\beta\tau}x(\tau)))^{2}\right)^{\frac{1}{2}}
\\&=O\left(\frac{1}{\sqrt{N}}\right)\end{aligned}\end{equation}Here, note that\begin{equation}\begin{aligned}&\mathbb{E}(y_i^{*})^{2}=\mathbb{E}(e^{-\beta\tau^{*}_i}x_i(\tau^{*}_i))^{2}
=(x^{*})^{2}\mathbb{E}(e^{-2\beta\tau^{*}_i})=(x^{*})^{2}e^{-\left(a'+\sqrt{(a')^{2}+4\beta}\right)b'}\\
&=(x^{*})^{2}e^{-\left(a'+\sqrt{(a')^{2}+4\beta}\right)\frac{1}{\sigma}\ln
\left(\frac{K}{\bar{\theta}_1}\frac{1}{x}\frac{k_2}{k_2-1}\right)}
\\&=\left(\frac{K}{\bar{\theta}_1}\frac{1}{x}\frac{k_2}{k_2-1}\right)^{-\frac{a'+\sqrt{(a')^{2}+4\beta}}{\sigma}}<+\infty.\end{aligned}\end{equation}When any alternative optimal stopping $\hat{\tau}_i \in \mathcal{S}$ is applied by $\mathcal{A}_i,$ then\begin{equation}\nonumber\begin{aligned}
&\mathcal{J}_i(\hat{\tau}_i,\tau_{-i}^{*})\\
=&\mathbb{E}\left\{\Bigg[\theta e^{-\beta\hat{\tau}_i}x_i(\hat{\tau}_i)+(1-\theta)\frac{e^{-\beta\hat{\tau}_i}x_i(\hat{\tau}_i)}
{l_1+l_2\left(\frac{1}{N}e^{-\beta\hat{\tau}_i}x_i(\hat{\tau}_i)+\frac{1}{N}\sum\limits_{j=1, j\neq i}^Ne^{-\beta\tau_j^{*}}x_j(\tau_j^{*})\right)}\Bigg]-e^{-\beta\hat{\tau}_i}K\right\}\\
=&E\left\{e^{-\beta\hat{\tau}_i}x_i(\hat{\tau}_i)\widetilde{\theta}-e^{-\beta\hat{\tau}_i}K\right\}.
\end{aligned}
\end{equation}Here, \begin{equation}\widetilde{\theta}(\hat{y}_i, y^{*}_{-i})=\theta +\frac{(1-\theta)}{l_1+\frac{l_2}{N}e^{-\beta\hat{\tau}_i}x_i(\hat{\tau}_i)+\frac{l_2}{N}\sum\limits_{j=1,j\neq i}^Ne^{-\beta\tau_j^{*}}x_j(\tau_j^{*})}.\end{equation}On the other hand, \begin{equation}J_{i}(\hat{\tau}_i)=\mathbb{E}\{\bar{\theta}_1\hat{y}_i-e^{-\beta \hat{\tau}_i}K\}.\end{equation}Applying Cauchy-Schwarz inequality, we have
\begin{equation}\nonumber\begin{aligned}
&|\mathcal{J}_i(\hat{\tau}_i, \tau_{-i}^{*})-J_{i}(\hat{\tau}_i)|=|\mathbb{E}\left[\hat{y}_i\left(\widetilde{\theta}(\hat{y}_i, y_{-i}^{*})-\bar{\theta}\right)\right]|\\
=&(1-\theta)\mathbb{E}\left[\hat{y}_i\left(\frac{1}{l_1+
l_2\frac{\hat{y}_i}{N}+l_2y_{-i}^{*}}-\frac{1}{l_1+l_2\mathbb{E}\Big(e^{-\beta\tau}x(\tau)\Big)}\right)\right]\\
\leq &(1-\theta)\left(\mathbb{E}(\hat{y}_i)^{2}\right)^{\frac{1}{2}}\cdot \mathbb{E}\left(\frac{l_2\frac{\hat{y}_i}{N}+l_2\big(y_{-i}^{*}-\mathbb{E}(e^{-\beta\tau}x(\tau))\Big)}{(l_1+
l_2\frac{\hat{y}_i}{N}+l_2y_{-i}^{*})(l_1+l_2\mathbb{E}\Big(e^{-\beta\tau}x(\tau)\Big))}^{2}\right]^{\frac{1}{2}}\\\leq & \frac{2(1-\theta)}{l_1^{2}}\left(\mathbb{E}(\hat{y}_i)^{2}\right)^{\frac{1}{2}}\cdot \left(\frac{l_2^{2}}{N^{2}}\mathbb{E}(\hat{y}_i)^{2}+l_2^{2}\mathbb{E}(y_{-i}^{*}-\mathbb{E}(e^{-\beta\tau}x(\tau)))^{2}\right)^{\frac{1}{2}}
\\=&O\left(\frac{1}{\sqrt{N}}\right)
\end{aligned}
\end{equation}Note that by Lemma 4.1, we have $\mathbb{E}(\hat{y}_i)^{2}<+\infty.$ Hence he result.
\end{proof}

\section{Inverse Mean Field Optimal Stopping Problem}
In this section, we turn to study the \emph{inverse} mean-field optimal stopping problem of large population system. Note that the inverse stopping problem in non-large-population setup is already well addressed (\cite{Kruse}) but to our best knowledge, no similar work was carried out in its large population setup.
Its main idea can be sketched as follows. We first introduce a market manager, denoted by $\mathcal{A}_0$, which can be interpreted as the supervisory authority, market regulator, or local government (say, the tax or revenue bureau). Unlike the individual negligible agents $\{\mathcal{A}_i\}_{i=1}^{N}$, the manager $\mathcal{A}_0$ has the right to construct or design the gain functional in our optimal stopping problem. One real example is the transaction fee $K$ introduced in \eqref{1e5}, which should be charged by the market organizer or local government (that is $\mathcal{A}_0$). Also, to great extent, the level of transaction fee should be designed or settled down subject to the discretion of $\mathcal{A}_0$. Given the transaction fee $K$ assigned by $\mathcal{A}_0,$ the small agents $\{\mathcal{A}_i\}_{i=1}^{N}$ in large population system naturally give their best stopping response $\{\tau_i^{*}\}_{i=1}^{N}$ to own individual optimal stopping problems, as we discussed in Section II and III. As a sequel, it generates an empirical distribution function
${F}_{N}(t):=\frac{1}{N}\sum_{i=1}^{N}1_{[0, t]}(\tau_i^{*})$ by all stopping times adopted by $\{\mathcal{A}_i\}_{i=1}^{N}$. By the Glivenko-Cantelli theorem, we have (see \cite{Vaart})$$||{F}_{N}-F||_{\infty}=\sup_{0 \leq t \leq +\infty}|{F}_{N}(t)-F(t)|\longrightarrow  0 \quad a.s. \quad N \longrightarrow +\infty$$where $F$ is the distribution function of $\tau_i^{*}, 1 \leq i \leq N.$ Recall the cutoff level of individual optimal stopping $\tau_i^{*}$ depends on $K$ thus $F_{N}, F$ are actually functions of $K$.

As the market organizer or regulator, $\mathcal{A}_0$ does not concern the individual stopping time applied by a particular agent. Instead, $\mathcal{A}_0$ is more interested to carefully design the transaction level $K$ such that the resulting empirical distribution $F_{N}$ or $F$ from all stopping times $\{\tau_i^{*}\}_{i=1}^{N}$ should possess some preferred statistical properties. In other words, the market manager is more concerned to the group stopping behavior by the large population system instead a given individual agent. For example, in some cases, the market organizer prefers the agents $\{\mathcal{A}_i\}_{i=1}^{N}$ will not close their business (or, sell their assets) all together in the same time thus it hopes to maximize the variance of empirical process by all stopping times. In other words, to maximize the empirical variance related to the distribution function $F_{N}$, or the variance to $F$ when considering the asymptotic property by $N\longrightarrow +\infty.$ To gain more insights, let us focus on the framework of \emph{best time to sell}, as addressed in Section (II). It can also be understood by the \emph{best time to close business}. The values of individual assets or business firms are given by\begin{equation}\nonumber \left\{
\begin{aligned}
&dx_i(t)=\alpha x_i(t)dt+\sigma x_i(t)dW_i(t),\\
&x_i(0)=x.
\end{aligned}\right.\end{equation}The individual gain or payoff functionals are:\begin{equation}\label{5e1}
\mathcal{J}_i(\tau_i,\tau_{-i})=\mathbb{E}\left\{\frac{ e^{-\beta\tau_i}x_i(\tau_i)}{\frac{1}{N}\sum\limits_{j=1, j \neq i}^Ne^{-\beta\tau_j}x_j(\tau_j)}- e^{-\beta\tau_i}K\right\}.
\end{equation}\begin{remark}The functional \eqref{5e1} can be viewed as a special case of the payoff functional \eqref{1e5} given in Problem (I) by setting $\theta=0, l_1=0, l_2=1.$ Moreover, in \eqref{5e1}, the state-average excludes the individual state of $\mathcal{A}_i.$ Similar functional but arising in linear-quadratic mean-field game can be found in \cite{hcm07}. We focus on the functional \eqref{5e1} here mainly because we aim to get more explicit result. It is remarkable that our analysis below can also be extended to more general functional but without explicit solutions.\end{remark}

Based on \eqref{nce1}, the consistent condition in present case \eqref{5e1} takes more explicit representation:$$\bar{\theta}_1=x^{-1}\Big(\frac{Kk_2}{k_2-1}\Big)^{1-\frac{\sigma}{\lambda_+}}=x^{-1}\Big(\frac{Kk_2}{k_2-1}\Big)^{1-\frac{1}{k_2}}.$$
For sake of presentation, we repeat the individual optimal stopping rules as follows: $\tau_i^{*} \sim \tau$ where\begin{equation}\label{e13} \left\{
\begin{aligned}
&\tau=\inf\{t:\ W(t)\geq a't+b'\},\\
&a'=\frac{\sigma}{2}-\frac{\alpha}{\sigma}, \quad b'=\frac{1}{\sigma}\ln\left(\frac{K}{\bar{\theta}_1x}\cdot\frac{k_2}{k_2-1}\right)=\frac{1}{\sigma}\ln\left(\frac{Kk_2}{k_2-1}\right)^{\frac{1}{k_2}}>0.
\end{aligned}\right.\end{equation}Note that $k_2>1$ thus $b'>0$ implies $\frac{Kk_2}{k_2-1}>1.$ That is, $K>\frac{k_2-1}{k_2}.$ Moreover, we have$$\mathbb{E}\Big(e^{(\lambda a'-\frac{1}{2}\lambda^2)\tau}\Big)=e^{-\lambda b'}.$$
Set $\lambda a'-\frac{1}{2}\lambda^2=t$, and we have\begin{equation}\label{e15}
M_{\tau}(t):=\mathbb{E}\Big(e^{t\tau}\Big)=e^{-\lambda b'}=e^{-\frac{\lambda}{\sigma}\ln\left(\frac{K k_2}{k_2-1}\right)^{\frac{1}{k_2}}}
=\left(\frac{Kk_2}{k_2-1}\right)^{-\frac{\lambda_+(t)}{k_2\sigma}}
\end{equation}where\begin{equation}\nonumber\left\{
\begin{aligned}
& \lambda_{+}(t)=\frac{\sigma}{2}-\frac{\alpha}{\sigma}+\sqrt{\Big(\frac{\sigma}{2}-\frac{\alpha}{\sigma}\Big)^2-2t},\\
&\lambda_{-}(t)=\frac{\sigma}{2}-\frac{\alpha}{\sigma}-\sqrt{\Big(\frac{\sigma}{2}-\frac{\alpha}{\sigma}\Big)^2-2t}.
\end{aligned}\right.\end{equation}Further, we have the following basic computations:
$$M_{\tau}'(t)=M_{\tau}(t)\cdot \ln \left(\frac{Kk_2}{k_2-1}\right)\cdot \frac{1}{k_2\sigma\sqrt{\Big(\frac{\sigma}{2}-\frac{\alpha}{\sigma}\Big)^2-2t}}.$$
It follows that\begin{equation}\nonumber\mathbb{E}(\tau)=M_{\tau}'(0)=\left(-\frac{1}{a'k_2\sigma}\right)\ln \left(\frac{Kk_2}{k_2-1}\right)\end{equation}which is an increasing function for $K > \frac{k_2-1}{k_2}.$ In fact, if $K=\frac{k_2-1}{k_2},$ then $x=x^{*},$ and the agents should sell their assets or close their business immediately at $t=0.$ Further, we have the following computations.\begin{equation}\nonumber\begin{aligned}
M_{\tau}''(t)&=M_{\tau}(t)\left[\ln \left(\frac{Kk_2}{k_2-1}\right)\cdot\frac{1}{k_2\sigma}\cdot \frac{1}{\sqrt{(a')^{2}-2t}}\right]^{2}\\
&+M_{\tau}(t)\left[\ln \left(\frac{Kk_2}{k_2-1}\right)\cdot\frac{1}{k_2\sigma}\cdot \frac{1}{(\sqrt{(a')^{2}-2t})^{3}}\right].
\end{aligned}
\end{equation}Therefore,\begin{equation}\nonumber\begin{aligned}
&\mathbb{E}\tau^{2}=M_{\tau}''(0)=\left[\ln \left(\frac{Kk_2}{k_2-1}\right)\cdot \left(\frac{1}{k_2\sigma}\right)\cdot \frac{1}{-a'}\right]^{2}-\ln \left(\frac{Kk_2}{k_2-1}\right)\cdot \left(\frac{1}{k_2\sigma}\right)\cdot \frac{1}{(a')^3}\\
\Longrightarrow &\text{Var}(\tau)=\mathbb{E}\tau^{2}-(\mathbb{E}\tau)^{2}=-\ln \left(\frac{Kk_2}{k_2-1}\right)\cdot \left(\frac{1}{k_2\sigma(a')^3}\right)\end{aligned}\end{equation}Actually, $\tau$ should follow the inverse Gaussian distribution $IG(\mu, \varrho)$ where\begin{equation}\nonumber\left\{
\begin{aligned}
& \mu=\left(-\frac{1}{a'k_2\sigma}\right)\ln \left(\frac{Kk_2}{k_2-1}\right)>0,\\
&\varrho=\left[\frac{1}{k_2\sigma}\ln \frac{Kk_2}{k_2-1}\right]^{2}>0.
\end{aligned}\right.\end{equation}It follows that the expectation and variance are both increasing functions of $K.$

\subsection{Target expectation or variance}

We first present a rather simple case which we called ``\emph{target expectation or variance}." In this case, the market manager $\mathcal{A}_0$ aims to push or steer
the empirical distribution $F_{N}$, or its limit $F$ to reach some given level in expectation or variance. Its economic interpretation is as follows: in some case, the market manager hopes the given small agents or firms will close their business with some prescribed average time, or some given variance. The target variance implies these small firms or agents in given industry will not close their business or sell their assets in too concentrated way (i.e. small variance) or dispersive way (i.e. large variance). To this problem, we have the following result.

\begin{proposition}To have $\mathbb{E}\tau=\int_{t=0}^{+\infty} t dF(t)=\mu_0>0,$ the transaction fee should be set by $K_{\mu_{0}}=\left(\frac{k_2-1}{k_2}\right)e^{-a'\mu_0k_2\sigma}$.\end{proposition}

\begin{proposition}To have $ \text{Var}(\tau)=\kappa_0>0,$ the transaction fee should be set by $K_{\kappa_{0}}=\left(\frac{k_2-1}{k_2}\right)e^{-(a')^{3}\kappa_0k_2\sigma}$.\end{proposition}

\begin{remark}(i) Note that the higher targeted expectation level $\mu_0$ or variance level $\kappa_0,$ the higher transaction level $K$ should be set.
(ii) Recall $a'<0, k_2>1$ thus $K_{\mu_{0}}, K_{\kappa_{0}}>\frac{k_2-1}{k_2}$ which is consistent with our previous result. \end{remark}

\subsection{Minimization of $L^{2}-$deviation of target time location}

Now we consider the case where the market manager $\mathcal{A}_0$ encourages the small agents to close their business (or sell their assets) at some pre-specified timing point $t_0.$ This corresponds to the situation when some local government plans to upgrading some declining industry in given future time in a very quick way. The empirical $L^{2}-$deviation with large population size $N$ is given by $\frac{1}{N}\sum_{i=1}^{N}(\tau_i^{*}-t_0)^{2}.$ Motivated by this situation, we consider the following $L^{2}-$deviation minimization problem:$$\min_{K}\mathbb{E}(\tau-t_0)^{2}.$$The following result is very straightforward.\begin{proposition}
The optimal transaction fee of the above minimization problem is given by$$\arg\min_{K}\mathbb{E}(\tau-t_0)^{2}=\begin{cases} \frac{k_2-1}{k_2}\exp\left(\frac{k_2\sigma}{2}\left(\frac{1}{a'}-2t_0a'\right)\right), \quad \quad t_0>\frac{1}{2 (a')^{2}};\\ \frac{k_2-1}{k_2}, \quad \quad t_0\leq \frac{1}{2 (a')^{2}}.\end{cases}$$
\end{proposition}

\begin{proof}The simple calculation gives that\begin{equation}\nonumber\begin{aligned}
&\mathbb{E}(\tau-t_0)^{2}=\mathbb{E}\tau^{2}-2t_0\mathbb{E}(\tau)+t_0^{2}\\
&=\frac{1}{(k_2\sigma a')^{2}}\left(\ln \frac{Kk_2}{k_2-1}\right)^{2}+\left(\frac{2t_0}{k_2\sigma a'}-\frac{1}{k_2\sigma (a')^{3}}\right)\left(\ln \frac{Kk_2}{k_2-1}\right)+t_0^{2}
\end{aligned}
\end{equation}Hence, if $\left(\frac{2t_0}{k_2\sigma a'}-\frac{1}{k_2\sigma (a')^{3}}\right) \geq 0 \Longleftrightarrow t_0 \leq \frac{1}{2 (a')^{2}},$ then $\mathbb{E}(\tau-t_0)^{2}$ becomes an increasing function of $K$ and the problem becomes trivial ($K=\frac{k_2-1}{k_2}.$) If $\left(\frac{2t_0}{k_2\sigma a'}-\frac{1}{k_2\sigma (a')^{3}}\right)<0 \Longleftrightarrow t_0>\frac{1}{2 (a')^{2}},$ then $\mathbb{E}(\tau-t_0)^{2}$ get its minimum at$$K=\frac{k_2-1}{k_2}e^{\frac{k_2\sigma}{2}\left(\frac{1}{a'}-2t_0a'\right)}.$$A direct check is as follows: note that $t_0>\frac{1}{2 a'^{2}}$ implies $\left(\frac{1}{a'}-2t_0 a'\right)>0$ thus the above optimal level $K>\frac{k_2-1}{k_2}$ which is consistent with our previous calculation.\end{proof}Now, we consider one related but more general case. Recall the transaction fee $K$ is charged by the market organizer thus it can be viewed as some revenue or income to $\mathcal{A}_0.$ Consequently, the large population optimal stopping problem discussed in Section II, will naturally generate some future cash flow $(\tau_i^{*}, K)_{i=1}^{N}$ to $\mathcal{A}_0.$ Therefore, besides the above concern of $L^{2}-$deviation minimization, the market manager can also consider how to maximize its utility to such cash flow:$$\sum_{i=1}^{N} \mathbb{E}\left(e^{-\beta \tau_i^{*}}U(K)\right)$$where $U$ is the given utility function. Recall $\beta>0$ is our discounted factor introduced before. Considering $N\longrightarrow +\infty$ so we study the cash flow utility but in per capital. Combined with the minimization of $L^{2}-$deviation, we reach the following more general optimization problem:\begin{equation}\label{mixed}\min_{K}\gamma_1\mathbb{E}(\tau-t_0)^{2}-\gamma_2\mathbb{E}e^{-\beta \tau}U(K)\end{equation}for $\gamma_1, \gamma_2>0.$ Here, the first term $\gamma_1\mathbb{E}(\tau-t_0)^{2}$ represents the concern to future distribution of stopping times while the second term $\gamma_2\mathbb{E}e^{-\beta \tau}U(K)$ is concerning to the future utility of cash flows. Note that the higher transaction fee $K$, the less incentive the individual agent will stop its own business at early time hence the less present value received by $\mathcal{A}_0.$ Therefore, the gain functional in \eqref{mixed} represents the trade-off between keeping future stopping as close to given timing point as possible, and the utility from future cash flow generated. The further calculation gives$$\mathbb{E}\left(e^{-\beta \tau}U(K)\right)=U(K)\mathbb{E}e^{-\beta \tau}=U(K)\left(\frac{x^{*}}{x}\right)^{-k_2}=\frac{U(K)}{K}\frac{k_2-1}{k_2}.$$The payoff functional becomes
\begin{equation}\nonumber\begin{aligned}
&\Gamma(K):=\gamma_1\mathbb{E}(\tau-t_0)^{2}-\gamma_2\mathbb{E}e^{-\beta \tau}U(K)\\
&=\frac{\gamma_1}{(k_2\sigma a')^{2}}\left(\ln \frac{Kk_2}{k_2-1}\right)^{2}+\gamma_1\left(\frac{2t_0}{k_2\sigma a'}
-\frac{1}{k_2\sigma (a')^{3}}\right)\left(\ln \frac{Kk_2}{k_2-1}\right)-\frac{U(K)}{K}\frac{\gamma_2 (k_2-1)}{k_2}+\gamma_1t_0^{2}.
\end{aligned}
\end{equation}The first order necessary condition becomes:$$\Gamma'(K)=-\frac{\gamma_2(k_2-1)}{k_2}\left(U'(K)-\frac{U(K)}{K}\right)+\gamma_1\left(\frac{2t_0}{k_2\sigma a'}-\frac{1}{k_2\sigma (a')^{3}}\right)+\frac{2\gamma_1}{(k_2\sigma a')^{2}}\ln \frac{Kk_2}{k_2-1}=0.$$

\subsubsection{Linear function}
One special case is $U(K)=K,$ and we have the same result in Proposition 5.3.
\subsubsection{Power utility function}
Another special case is $U(K)=\frac{K^{\rho}}{\rho}$ for $\rho<1,$ and we have$$\Gamma'(K)=\Theta(K)+\gamma_1\left(\frac{2t_0}{k_2\sigma a'}-\frac{1}{k_2\sigma (a')^{3}}\right)=0$$where\begin{equation}\nonumber
\Theta(K)=-\gamma_2\left(1-\frac{1}{\rho}\right)\frac{k_2-1}{k_2}K^{\rho-1}+\frac{2\gamma_1}{(k_2\sigma a')^{2}}\ln \frac{Kk_2}{k_2-1}
\end{equation}Note that\begin{equation}\nonumber\left\{
\begin{aligned}
&\Gamma''(K)=\Theta'(K)=\frac{2\gamma_1}{(k_2\sigma a')^{2}}\frac{1}{K}-\gamma_2\left(\frac{k_2-1}{k_2}\right)\frac{(1-\rho)^{2}}{\rho}K^{\rho-2}\\
& \lim_{K\longrightarrow\frac{k_2-1}{k_2}}\Theta(K)=\gamma_2\left(\frac{1}{\rho}-1\right)\left(\frac{k_2-1}{k_2}\right)^{\rho}<+\infty,\\
&\lim_{K\longrightarrow +\infty}\Theta(K)=+\infty.
\end{aligned}\right.\end{equation}Therefore, we have the following two cases. For simplification, we introduce the following notations:

\begin{equation}\nonumber\left\{
\begin{aligned}
&\Delta_1=\left(\frac{2\gamma_1}{\gamma_2}
\frac{1}{(k_2\sigma a')^{2}}\frac{\rho}{(1-\rho)^{2}}\right)^{\frac{1}{2-\rho}}\\
&\Delta_2=\gamma_1\left(\frac{2t_0}{k_2\sigma a'}-\frac{1}{k_2\sigma (a')^{3}}\right)+\gamma_2\left(\frac{1}{\rho}-1\right)\left(\frac{k_2-1}{k_2}\right)^{\rho}.
\end{aligned}\right.\end{equation}

\textbf{Case 1}: if $\left(\frac{2\gamma_1}{\gamma_2}
\frac{k_2}{k_2-1}\frac{1}{(k_2\sigma a')^{2}}\frac{\rho}{(1-\rho)^{2}}\right)^{\frac{1}{1-\rho}}
< \frac{k_2-1}{k_2} \Longleftrightarrow \Delta_1<\left(\frac{k_2-1}{k_2}\right),$ then we have $\Gamma''(K)>0$ for $K \in (\frac{k_2-1}{k_2}, +\infty),$ thus it follows$$\arg\min_{K}\Gamma(K)=\begin{cases} \frac{k_2-1}{k_2}, \quad \text{if} \quad \Delta_2>0;
\\ \text{unique solution of}\quad \Gamma'(K)=0, \quad \text{if} \quad \Delta_2<0.\end{cases}$$

\textbf{Case 2}: if $\left(\frac{2\gamma_1}{\gamma_2}
\frac{k_2}{k_2-1}\frac{1}{(k_2\sigma a')^{2}}\frac{\rho}{(1-\rho)^{2}}\right)^{\frac{1}{1-\rho}}
\geq \frac{k_2-1}{k_2} \Longleftrightarrow \Delta_1 \geq \left(\frac{k_2-1}{k_2}\right),$ then we have $\Gamma''(K)<0$ for $K \in \left(\frac{k_2-1}{k_2}, \Delta_1\right),$ and $\Gamma''(K)<0$ for $K \in \left(\Delta_1, +\infty\right)$, thus it follows$$\arg\min_{K}\Gamma(K)=\begin{cases} \frac{k_2-1}{k_2}, \quad \text{if} \quad \Delta_2 >0, \Theta(\Delta_1)>0;
\\ \text{the larger root of}\quad \Gamma'(K)=0, \quad \text{if} \quad \Delta_2 >0, \Theta(\Delta_1)<0;\\
\text{the unique root of}\quad \Gamma'(K)=0, \quad \text{if} \quad \Delta_2 <0.\end{cases}$$

\subsection{Minimization of Kullback-Leibler (KL) divergence}

Now, we turn to consider the case when the market manager aims to have the empirical distribution $F_{N}$ of all stopping times, or its limit $F$ to best fit some given benchmark probability distribution $\pi(t)$ with support on $(0, +\infty)$. Moreover, assume $\pi$ is absolute continuous with respect to Lebesgue measure with density function $p(t)$. The economic meaning of such
distribution tracking is as follows. Unlike \emph{Case A, B} (namely, \emph{target expectation/variance}, or $L^{2}-$\emph{minimization of timing point}) where the tracking is mainly in its static manner, sometimes the market manager hopes to have the tracking in more dynamic way. Actually, it is often the case that $\mathcal{A}_0$ hopes the stopping times will best fit into some distribution in given industries to match some long-term planning. For example, for some sunset or declining industries, the local government hopes to upgrading it in some progressive and orderly manner. Therefore, it is hopeful the individual small firms in this industry will close their business distributed by some given prescribed distribution function $p(t)$ to coordinate the corresponding fiscal or employment policies.
Therefore, it aims to $\min_{K \in (\frac{k_2-1}{k_2}, +\infty)} ||\frac{1}{N}\sum_{i=1}^{N}\delta_{\tau_i}(0, t)-\pi(t) ||$ where $\delta$ is the Dirac measure. Alternatively, $\min_{K \in (\frac{k_2-1}{k_2}, +\infty)} ||F-\pi||$ in its limiting case when $N\longrightarrow +\infty$. Note that we require $K \in (\frac{k_2}{k_2-1}, +\infty)$ to make the optimal stopping problem in Section III nontrivial.

For sake of simplification, we focus on the limiting case in terms of the density function. Given the target density function $p(t),$ we can measure the performance of best fitting by the Kullback-Leibler (KL) divergence.\begin{definition}The Kullback-Leibler (KL) divergence of $q(t)$ from the targeted density function $p(t)$ is given by$$\mathcal{D}(p||q)=\int \ln \frac{p(t)}{q(t)}\cdot p(t)dt$$\end{definition}It is remarkable that $\mathcal{D}(p||q) \geq 0$ due to Jensen's inequality and $\mathcal{D}(p||q)=0 \Longleftrightarrow p=q$ due to Gibbs'inequality. Moreover, it follows that in general $\mathcal{D}(p||q) \neq \mathcal{D}(q||p)$ thus KL divergence is not distance. We have the following dynamic optimal tracking problem from large-population system.

\textbf{Minimization of KL-divergence.}
To find $K \in (\frac{k_2-1}{k_2}, +\infty)$ satisfying
\begin{equation}\nonumber\begin{aligned}
K \in \arg\min \mathcal{D}_{K}(p||q)
\end{aligned}
\end{equation}where $q(t)$ is the density function of Inverse Gaussian distribution $IG(\mu, \varrho)$ for\begin{equation}
\begin{aligned}
\mu=\left(-\frac{1}{a'k_2\sigma}\right)\ln \left(\frac{Kk_2}{k_2-1}\right), \quad \varrho=\left[\frac{1}{k_2\sigma}\ln \frac{Kk_2}{k_2-1}\right]^{2}.
\end{aligned}\end{equation}Explicitly, the density function of inverse Gaussian distribution is given by\begin{equation}\nonumber\begin{aligned}q(t)&=\left[\frac{\varrho}{2\pi t^{3}}\right]^{\frac{1}{2}}\exp \left(\frac{-\varrho(t-\mu)^{2}}{2\mu^{2}t}\right)\\&=\frac{\ln\left(\frac{Kk_2}{k_2-1}\right)}{k_2\sigma(2\pi t^{3})^{\frac{1}{2}}}\cdot \exp\left[-\frac{(a')^{2}}{2}\frac{(t-\mu)^{2}}{t}\right]\\&=\frac{\ln\left(\frac{Kk_2}{k_2-1}\right)}{k_2\sigma(2\pi t^{3})^{\frac{1}{2}}}\cdot \exp\left[-\frac{(a')^{2}}{2}\frac{\left(t+\frac{1}{a'k_2\sigma}\ln (\frac{Kk_2}{k_2-1})\right)^{2}}{t}\right].\end{aligned}\end{equation}
One algorithm based on the empirical
risk minimization (ERM) (see \cite{Vapnik}) is as follows:\begin{equation}\nonumber\begin{aligned}
& \widehat{K}=\arg \min_{K \in (\frac{k_2-1}{k_2}, +\infty)}\mathcal{D}_{K}(p||q)=\arg \min \int \ln \frac{p(t)}{q(t)}\cdot p(t)dt\\
&=\arg \min \left(-\int \ln q(t)\cdot p(t)dt\right)\\
&\approx \arg \max \int \ln q(t)\left[\frac{1}{n}\sum_{i=1}^{n}\delta(t-t_i) \right]dt\\&=\arg \max \frac{1}{n}\sum_{i=1}^{n}\ln q(t_i)=\arg \max \frac{1}{n}\prod_{i=1}^{n}q(t_i)
\end{aligned}\end{equation}which leads to the maximum likelihood estimator based on the sample points $\{t_i\}_{i=1}^{n}$ from distribution $p(t).$ Note that the definition of risk function (see \cite{Vapnik}):
$$R(K)=-\int\ln q(t)p(t)dt$$ and we have the following identity\begin{equation}\nonumber\begin{aligned}
\mathcal{D}_{K}(p||q)&=\int \ln \frac{p(t)}{q(t)}\cdot p(t)dt\\
&=R(K)+\int \ln p(t)p(t)dt=R(K)+\text{Entropy of $p$}
\end{aligned}\end{equation}Recall the benchmark density function $p$ does not depend on the selection of $K$ hence the minimization of KL-divergence is equivalent to the minimization of risk function. Moreover, the justification of applying KL-divergence as our performance measure can also be found by the following Bretagnolle-Huber inequality (see \cite{Vaart}):$$||p-q||_{1}=\int |p(t)-q(t)|dt \leq 2\sqrt{1-\exp(-\mathcal{D}_{K}(p||q))}\leq \sqrt{2\mathcal{D}_{K}(p||q)}.$$That is, the smaller value of KL-divergence implies a smaller $L_{1}-$norm.

\section{Conclusion}
In this paper, we introduce and analyze two classes of large population optimal stopping problems by considering the relative performance. The consistency condition and $\epsilon-$Nash equilibrium are established. We also discuss the inverse mean-field optimal stopping problem where the transaction fee $K$ is designed to meet some preferred statistical property of all stopping times generated.
Our present paper suggests various future research directions. For example, it is possible to introduce and study the more general inverse mean-field optimal stopping: that is, to find a time-dependent function $\pi: [0, \infty)\longrightarrow \mathbb{R}$ such that the decentralized optimal stopping times $\{\tau^{*}_i\}_{i=1}^{N}$ for
$$ \mathbb{E}\left[g\left(\tau_i, x_{i}(\tau_i), x^{N}(\tau_i)\right)+\pi(\tau_i))\right]$$will satisfy some given statistical property. Here, $x^{N}(\cdot)$ denotes some term characterizing the state-average, $\pi(\cdot)$ is called the transfer function (see e.g. \cite{Kruse}). Another example is to introduce some dynamics of market manager $\mathcal{A}_0$ and thus consider the optimal stopping problem with major-minor agents.

\end{document}